\documentclass[12pt,reqno]{amsart}

\usepackage{latexsym}
\usepackage{amssymb}
\usepackage{graphics}
\usepackage{graphicx}
\usepackage{color}

\renewcommand{\epsilon}{\varepsilon}
\newcommand{\dist}{{\operatorname{dist}}}

\newcommand{\kahler}{K\"ahler }

\newcommand{\PP}{{\mathbb P}}

\newcommand{\R}{{\mathbb R}}
\newcommand{\C}{{\mathbb C}}

\newcommand{\CP}{\C\PP}

\newcommand{\dbar}{\bar\partial}
\newcommand{\ddbar}{\partial\dbar}

\newcommand{\E}{{\mathbf E}}

\newcommand{\Zb}{{\mathbf Z}}

\newcommand{\epk}{\epsilon(k)}

\renewcommand{\phi}{\varphi}

\newcommand{\hcal}{\mathcal{H}}

\newcommand{\ocal}{\mathcal{O}}

\newtheorem{theo}{{Theorem}}[section]
\newtheorem{cor}[theo]{{Corollary}}
\newtheorem{lem}[theo]{{Lemma}}
\newtheorem{prop}[theo]{{Proposition}}

\newenvironment{rem}{\medskip\noindent{\it Remark:\/} }{\medskip}

\title[partial density function]{asymptotics of partial density function vanishing along smooth subvariety}

\author{Jingzhou Sun }
\address{Department of Mathematics, Shantou University, Shantou City, Guangdong
	515063,China} \email{jzsun@stu.edu.cn}
\thanks{The author is partially supported by NNSF of China no.11701353 and the STU Scientific Research Foundation for Talents no.130/760181.}

\date{\today}

\begin{document}
\begin{abstract}
	We study the asymptotic of the partial density function associated to holomorphic section of a postive line bundle vanishing to high orders along a fixed smooth subvariety. Assuming local torus-action-invariance, we describe the forbidden region, generalizing the result of Ross-Singer. 
\end{abstract}

\maketitle

\tableofcontents
	\section{Introduction}	
Let $(X,L)$ be a polarized compact \kahler manifold of dimension $m$. We endow $L$ with a Hermitian metric $h$ with positive curvature. And we use $\omega=\frac{i}{2}\Theta_h$ as the \kahler form. By abuse of notation, we still use $h$ to denote the induced metric on the $k$-th power $L^k$. Then we have a Hermitian inner product on $H^0(X,L^k)$, defined by
$$<s_1,s_2>=\int_X h(s_1,s_2)\frac{\omega^m}{m!}$$
Let $\{s_i\}$ be an orthonormal basis of $\hcal_k=H^0(X,L^k)$. Then the  on-diagonal Bergman kernel, also called density of states function, is defined as
$$\rho_k(z)=\sum |s_i(z)|_h^2.$$

 The  on-diagonal Bergman kernel $\rho_k$ has very nice asymptotic expansion by
the results of Tian, Zelditch, Lu, etc. \cite{Tian1990On, Zelditch2000Szego, Lu2000On, Catlin, MM}. The asymptotic expansion has been found very useful, making $\rho_k$ an important bridge connecting \kahler geometry to algebraic geometry (see for example \cite{donaldson2001,DonaldsonSun2}).

Let $V\subset X$ be a closed submanifold of dimension $n$.
Let $\delta>0$ be a small number. We denote by $\hcal_k^{\delta k}\subset \hcal_k$ the subspace consisting of holomorphic sections whose vanishing orders along $V$ are no less than $\delta k$. The partial density function $\rho_k^\delta$ is then defined as the density of states function for $\hcal_k^{\delta k}$: 
$$\rho_k^\delta(z)=\sum |s_i(z)|^2_h ,$$
	where $\{s_i\}$ is an orthonormal basis of $\hcal_k^{\delta k}$. When $V$ is a divisor,  the partial density functions have been studied by Zelditch-Zhou \cite{ZelditchZhou-Interface,ZelditchZhou2019}, Ross-Singer\cite{Ross2017} and Coman-Marinescu \cite{CM17} etc. It was first shown by Shiffman-Zelditch \cite{SZ2004} in the toric case, that there is a `forbidden region' in which the partial density function is exponentially small. Berman \cite{Berman2007} showed that there is an open forbidden region $R$ that asymptotically the partial density function is exponentially small on compact subsets and is equal to the usual density function on compact subsets of the complement of $\bar{R}$.  
		In particular,  Ross-Singer\cite{Ross2017} studied the case when $V$ is smooth divisor, plus the condition that the line bundle and the metric is $S^1$-symmetric in a neighborhood $U$ of $V$. Let $\mu:U\to \R$ be the Hamiltonian of the $S^1$-action, normalized so that $\mu^{-1}(0)=V$. Then Ross-Singer proved a distributional asymptotic expansion for the partial density function and showed that the forbidden region is $\mu^{-1}[0,\delta)$:
	\begin{theo}[\cite{Ross2017}]
		For sufficiently small $\delta$, we have
		$$\rho_{k}^\delta \sim \left\{
		\begin{array}{lcl}
		O(k^{-\infty})& &\text{on } \mu^{-1}[0,\delta) \\
		\rho_k+O(k^{-\infty})& & \text{on } X\backslash \mu^{-1}[0,\delta]
		\end{array}
		\right.,$$
		where $\sim$ means  the equality holds on any given compact subset of $\mu^{-1}[0,\delta)$ and $X\backslash \mu^{-1}[0,\delta]$ respectively. 
	\end{theo}
The term $k^{-\infty}$ in the theorem means a quantity depending on $k$ that is $O(k^{-m})$ as $k\rightarrow\infty$, for all $m\geq0$. 
	We will adopt the notation used by Donaldson and denote $k^{-\infty}$ by $\epsilon (k)$. So, in order to avoid confusion, we use $\delta$ for the partial density function instead of $\epsilon$ used in \cite{Ross2017}. 

\
	
In this article, we generalize Ross-Singer's result to the case when $V$ is a smooth subvariety. Let	$T^{m-n}=(S^1)^{m-n}$ denote the $(m-n)$-dimensional torus. 
We assume that there is a neighborhood $U$ of $V$ that admits a holomorphic $T^{m-n}$-action. We also assume the line bundle $L$ and the metric $h$, hence $\omega$, are invariant under the $T^{m-n}$-action.

Let $\mu: U\to \R^{m-n}$ be the moment map of the torus action, normalized so that $\mu(V)=0$. We denote by $\mu_i$ the $i$-th component of $\mu$. 
Then $\nu=\sum_{i=1}^{m-n}\mu_i:U\to \R$ is the moment map for the diagonal $S^1$-action.
Our main result is that the forbidden region is $\nu^{-1}([0,\delta))$:

\begin{theo}\label{main}
	For sufficiently small $\delta$, we have
	$$\rho_{k}^\delta \sim \left\{
	\begin{array}{lcl}
	\epk& &\text{on } \nu^{-1}[0,\delta) \\
	\rho_k+\epk& & \text{on } X\backslash \nu^{-1}[0,\delta]
	\end{array}
	\right.$$
\end{theo}
The proof uses the techniques we developed in \cite{SunSun,Annalen2019,Punctured}.
The idea of the proof is similar to that in \cite{Ross2017}. Unfortunately, we have not been able to give an asymptotic expansion for the renormalized partial Bergman kernel as Ross-Singer did in \cite{Ross2017}. It is also very interesting to know if there is still a forbidden region and what the boundary is, if the assumption of the torus-invariance is removed.

We would like to compare the partial density function to the logarithmic Bergman kernel studied by the author in \cite{sun2019logarithmic}, for which the vanishing order is $\geq 1$ instead of $\geq \delta k$. In \cite{sun2019logarithmic}, we showed that the asymptotics near the subvariety depend on the distance function to the subvariety. When $\delta$ is small enough, the moment function $\nu$ is comparable with the distance function squared. One can see the similarity between them.

 The asymptotics of off-diagonal Bergman kernel have been extensively used in the  value distribution theory of sections of line bundles by Shiffman-Zelditch and others\cite{bsz0,bsz1,bsz2,bsz3,sz4,sz1,sz2,sz3,SZ2004,szz,sz6,dsz1,dsz2,dsz3}. It would be very interesting and important to study the asymptotics of the partial off-diagonal Bergman kernel vanishing along a general subvariety.

	\
	
The structure of this article is as follows. We will first prove the key estimates on the complex plane. Then we apply the key lemma to the neighborhood $U$. Then with the help of the H\"{o}rmander's technique, we show that what we get on $U$ goes directly to $X$, hence proving the main theorem.

\

\textbf{Acknowledgements.} The author would also like to thank Professor Song Sun and Professor Chengjie Yu for many very helpful discussions. 

	\section{One dimensional model}


Let $\omega=f\sqrt{-1}dz\wedge d\bar{z}$ be a smooth \kahler form on $\C$ invariant under the natural $S^1$-action. Let $\phi(|z|)$ be a smooth potential of $\omega$. By a linear change of coordinates, we can assume that $\phi=|z|^2+O(|z|^4)$. Let $\mu:\C\to\R$ be the moment map of the $S^1$-action normalized so that $\mu(0)=0$. We will use the notations $r=|z|$, $x=|z|^2$ and $t=-\log x$. So $\phi_{tt}=x\Delta \phi=xf$.
Then $$\mu=\int_0^x fdx.$$
And since $\lim_{t\to\infty}\phi_{t}=0$, we have
\begin{eqnarray*}
	\phi_t&=&-\int_{t}^{\infty}\phi_{tt}dt\\
	&=&\int_{x}^{0}fdx\\
		&=&-\mu
\end{eqnarray*}
Let $a\geq 0$ be an integer, $f_1(|z|)>0$ be a smooth function, we consider the $k$-th norm of the holomorphic function $z^a$:
$$I_a=\int_{\C}e^{-k\phi}|z|^{2a}f_1\sqrt{-1}dz\wedge d\bar{z}=2\pi\int_{-\infty}^{\infty}e^{-k\phi-(a+1)t+\log f_1}dt.$$
Denote by $g_a(t)=-k\phi-(a+1)t+\log f_1$. Then $g_a'=-k\phi_t-(a+1)+(\log f_1)_t$ and $g_a''=-k\phi_{tt}+(\log f_1)_{tt}$. Since $f_1=A+O(|z|^2)$, $A>0$, we have $(\log f_1)_{t}=O(\phi_{t})$ and $(\log f_1)_{tt}=O(\phi_{tt})$. So for $k$ large enough, $g_a$ is a concave function of $t$ which attains its only maximum at $t_a$ satisfying
$$\phi_t(t_a)=-\frac{a+1}{k}+\frac{(\log f_1)_t(t_a)}{k}.$$
Namely, $$\mu(t_a)=\frac{a+1}{k}+O(\frac{x_a}{k}),$$ 
when $x_a<1$.

\

We recall the following basic lemma used in \cite{SunSun}:

\begin{lem}\label{lemconcave}
	Let $f(x)$ be a concave function. Suppose $f'(x_0)<0$, then we have
	$$\int_{x_0}^\infty e^{f(x)}dx\leq\frac{e^{f(x_0)}}{-f'(x_0)}$$
\end{lem}

Let $0<R_1<R<1$ be two fixed numbers.
 When $\mu(t_a)>R$, we have $a>Rk-C$ for some $C$ independent of $k$, then $g_a'<0$ when $\mu\leq R$. Let $\mu(t_R)=R$ and $\mu(t_{R_1})=R_1$, then we let $t_m=\frac{t_R+t_{R_1}}{2}$. Clearly, $|g_a'(t_m)|>C_1k$ for some $C_1>0$ depending on $R$ and $ R_1$ but independent of $k$ and $a$. By the concavity of $g_a$, we have
 $$e^{g_a(t_m)-g_a(t_{R_1})}=\epk$$
 	Therefore, by lemma \ref{lemconcave}, we have
 	$$\int_{\mu<R_1}e^{g_a(t)}dt=\epk \int_{R_1\leq \mu\leq R}e^{g_a(t)}dt.$$
 	
 	When $\mu(t_a)\leq R$, depending on $a$, there are two cases . 
  
 	When $a\geq \sqrt{k}$, $g_a''(t_a)$ is large, so we can use Laplace's method to estimate the integral, namely
 	$$I_a\approx \frac{(2\pi)^{3/2}}{\sqrt{g_a''(t_a)}}.$$
 	More importantly, the  mass of the measure $e^{g_a(t)}dt$ is concentrated in a small neiborhood of $t_a$. More precisely, if $k\phi_{tt}(t_a)\delta_a^2>2\log^2 k$, then by lemma \ref{lemconcave}, we have $$I_a=(1+\epk)\int_{t_a-\delta_a}^{t_a+\delta_a}e^{g_a(t)}dt.$$
 	It is easy to see that there is a constant $C_2>0$ independent of $k$ such that if $\delta_a>C_2\frac{\log k}{\sqrt{a}}$, then the condition $k\phi_{tt}(t_a)\delta_a^2>2\log^2 k$ is satisfied.

	When $a<\sqrt{k}$, the term $k\phi_{tt}(t_a)$ may be too small to use Laplace's method. We have the following
	
	\begin{eqnarray*}
		\int_{x<k^{-1/4}} e^{g_a(t)}dt&\geq& e^{-(a+1)}e^{-(a+1)e}e^{g_a(t_a)}\\
		&\geq& e^{-(a+1)}e^{-(a+1)e}e^{-(a+1)+a\log\frac{a+1}{k}}\\
		&\geq& e^{-5(a+1)}(\frac{a+1}{k})^a
	\end{eqnarray*}
When $x=k^{-1/4}$, $\frac{e^{g_a(\log k^{1/4})}}{e^{-5(a+1)}(\frac{a+1}{k})^a}=\epk$. So by lemma 
\ref{lemconcave}, we have for $a<\sqrt{k}$,
$$\int_{\mu\geq 2k^{-1/4}} e^{g_a(t)}dt	\leq \int_{x\geq k^{-1/4}} e^{g_a(t)}dt=\epk I_a$$

In summary, we have proved the following
\begin{lem}\label{1-dim-model}
For fixed $0<R_1<R<1$, there exist constants $C, C_1>0$ independent of $k$, such that for $k$ large enough we have the following estimations:
	\begin{itemize}
		\item[1)] If $a>Rk-C$, $$\int_{\mu<R_1}e^{g_a(t)}dt=\epk \int_{R_1\leq \mu\leq R}e^{g_a(t)}dt,$$
		where the $\epk$ term is independent of $a$.
			\item[2)] If $\sqrt{k}\leq a\leq Rk-C$,$$\int_{|\mu-\frac{a}{k}|\geq C_1\frac{\sqrt{a}\log k}{k}}e^{g_a(t)}dt=\epk \int_{|\mu-\frac{a}{k}|< C_1\frac{\sqrt{a}\log k}{k}}e^{g_a(t)}dt,$$
			where the $\epk$ term is independent of $a$.
				\item[3)] If $a<\sqrt{k}$, $$\int_{\mu\geq 2k^{-1/4}} e^{g_a(t)}dt	=\epk \int_{\mu< 2k^{-1/4}} e^{g_a(t)}dt,$$
				where the $\epk$ term is independent of $a$.
	\end{itemize}
	
\end{lem}

	\section{proof of the main theorem}
By assumption, there is a neighborhood $U$ of $V$ that is covered with charts that admit standard coordinates $(z,w)$ so that $V=\{z=0\}$, where $z=(z_1,\cdots,z_{m-n})\in \C^{m-n}$ and $w=(w_1,\cdots,w_n)\in \C^n$, and
$$\lambda\cdot(z,w)=(\lambda z, w),$$
for $\lambda=(\lambda_1,\cdots,\lambda_{m-n})$, where $\lambda z=(\lambda_1z_1,\cdots,\lambda_{m-n}z_{m-n})$. We have the following lemma, 
\begin{lem}
	There is a $T^{m-n}$-invariant biholomorphism $\Psi$ from the neighborhood $U$ to an open neighborhood of $V$ in the total space of the vector bundle $L_1\oplus L_2\cdots \oplus L_{m-n}$, where each $L_i$ is a holomorphic line bundle.
\end{lem}
 The proof is very similar to that of lemma 5.3 in \cite{RossSinger}, and for the convenience of the reader, we include a proof here.
\begin{proof}
	Let $(z_{\alpha},w_{\alpha})$ and $(z_\beta,w_\beta)$ be two sets of standard coordinates. Then the transition functions between them are necessarily of the form
	$$z_\beta=\lambda_{\alpha\beta}(z_\alpha,w_\alpha)z_\alpha,\quad w_\beta=\tau_{\alpha\beta}(z_\alpha, w_\alpha) $$
	where $\lambda_{\alpha\beta}$ and $\tau_{\alpha\beta}$ are holomorphic and take values in $\C^{*(m-n)}$ and $\C^{(m-n)}$ respectively.
	We understand the torus $T^{m-n}$-action on the $w$-part is trivial. Then both $\lambda_{\alpha\beta}$ and $\tau_{\alpha\beta}$ are $T^{m-n}$-invariant, so their dependence on $z_\alpha$ are trivial. So the transition functions
	$$z_\beta=\lambda_{\alpha\beta}(w_\alpha)z_\alpha$$
	defines a holomorphic vector bundle $E$ of rank $m-n$ over $V$, with structure group $T^{m-n}$. Therefore, $E$ decomposes as $E=L_1\oplus L_2\cdots\oplus L_{m-n}$, and $U$ is biholomorphic to an open neighborhood of $V$ in $E$. Clearly, the biholomorphism is $T^{m-n}$-invariant.
\end{proof}
We consider $U$ as an open set in $\oplus L_i$ and denote by $\pi:U\to V$ the projection. Then for $A$ small enough, $U$ contains the poly-disc bundle $U_A=\cap_{1\leq i\leq m-n}\mu_i^{-1}(x<A)$. For each $p\in V$, the fiber of $U_A$ is an open poly-disc $D_A$ in $\C^{m-n}$. By the torus-invariance, the restriction $\phi_p$ of $\phi$ depends only on $|z_1|,\cdots,|z_{m-n}|$, so does the restriction $\omega_p$ of $\omega$. So the monomials $z^a=\prod z_i^{a_i}$ are all orthogonal to each other on $D_p$, where  $a=(a_1,\cdots,a_{m-n})$ is an $(m-n)$-tuple of non-negative integers.

\begin{lem}
	Let $W_\delta=\{q\in D_p|\nu(q)<\delta-k^{1/4} \}$. 
	When $\delta$ is small enough, for each holomorphic function $f=\sum_{\sum a_i\geq \delta k}c_az^a$, we have
	$$\int_{W_\delta}|f|^2e^{-k\phi_p}\omega_p^{m-n}=\epk\int_{D_p}|f|^2e^{-k\phi_p}\omega_p^{m-n}$$
\end{lem}
\begin{proof}
	By orthogonality, $\int |f|^2e^{-k\phi_p}\omega_p^{m-n}=\sum_{\sum a_i\geq \delta k} |c_a|^2\int |z^a|^2 e^{-k\phi_p}\omega_p^{m-n}$. So it suffices to prove the lemma for the monomials $z^a$ satisfying $\sum a_i\geq \delta k$. 
	
	By making $A$ a little smaller, a compactness argument shows that the constants $C$ and $C_1$ in lemma \ref{1-dim-model} can be chosen to be the same for all $z_i$-discs. For each point $p$ in the open set , at least one $1\leq i\leq m-n$ satisfies the condition that $\mu_i< \frac{a_i}{k}-k^{-3/8}$. Therefore, lemma \ref{1-dim-model} shows that the mass of $z^a$ in $W$ is $\epk$ relative to the mass in $D_A\backslash W_\delta$. And the theorem is proved

\end{proof}
As a corollary, we can prove the part about the forbidden region of theorem \ref{main}.
\begin{proof}[Proof of the first half of theorem \ref{main}]
	We can take an arbitrary orthonormal basis $\{s_i\}$ for the space $\hcal^{\delta k}_k$. Then the lemma above shows that each $s_i$ has $L^2$-mass of the size $\epk$ in the region $W_\delta$. We reall that for each point $p\in X$, there is an unit section $s_p$, called the peak section, satisfying $|s_p(p)|^2_h=\rho_k(p)$ and $s_p\perp s$ as long as $s(p)=0$.

We take the inner product of $s_i$ with the peak section$s_p$ for each $p\in W$
	$$|s_i(p)|=|<s_i,s_p>|\sqrt{\rho_k(p)}\leq \parallel s_i\parallel \sqrt{\rho_k(p)} $$
	 we see that the point-wise norm $|s_i(p)|$ is also $\epk$.   Since the dimension of $\hcal^{\delta k}_k$ is $O(k^m)$. We get the conclusion.  
\end{proof}

We denote by $T_i$ the $i$-th component of $T^{m-n}$. 
For each $L_i$, we fix a $T_i$-invariant section $s_i$ such that $\mu_i^{-1}(0)=\{s_i=0\}$. So we can choose local frame $e_i$ of $L_i$, so that $s_i=z_i e_i$ in the standard torus-invariance coordinates. 

For each $s\in H^0(V,kL-\otimes L_i^{a_i})$, where $a_i\geq 0$, $\pi^*s\in H^0(U,kL-\otimes L_i^{a_i})$, so $\tilde{s}=\pi^*s\otimes\prod s_i^{a_i}\in H^0(U,L^k)$. 
 We denote by $G_\delta$ the span of all such sections for $\sum_{i=1}^{m-n}a_i<\delta k$. Then clearly, for each section $\beta\in \hcal^{\delta k}_k$, we have $\beta|_U\perp G_\delta$. Moreover, since each $\tilde{s}$ is represented by a holomorphic function of the form
 $$\sum_{\sum_{i=1}^{m-n}a_i<\delta k} f_a(w)z^a,$$
  by lemma \ref{1-dim-model}, each section $\tilde{s}\in G_\delta$ has its mass concentrated within  the set $Y_\delta=\{q|\nu(q)\leq \delta+2(m-n)k^{-1/4} \}$ and decays fast as $\nu$ gets bigger. More precisely, the measure outside $Y_\delta$ is $\epk$ relative to that inside $Y_\delta$. To connect $G_\delta$ to $(\hcal^{\delta k}_k)^\perp$, we need to use H\"ormander's $L^2$ technique. The following lemma is well-known, see for example \cite{Tian1990On}. 

\begin{lem}
	Suppose $(M,g)$ is a complete \kahler manifold of complex dimension $n$, $\mathcal L$ is a line bundle on $M$ with hermitian metric $h$. If 
	$$\langle-2\pi i \Theta_h+Ric(g),v\wedge \bar{v}\rangle_g\geq C|v|^2_g$$
	for any tangent vector $v$ of type $(1,0)$ at any point of $M$, where $C>0$ is a constant and $\Theta_h$ is the curvature form of $h$. Then for any smooth $\mathcal L$-valued $(0,1)$-form $\alpha$ on $M$ with $\bar{\partial}\alpha=0$ and $\int_M|\alpha|^2dV_g$ finite, there exists a smooth $\mathcal L$-valued function $\beta$ on $M$ such that $\bar{\partial}\beta=\alpha$ and $$\int_M |\beta|^2dV_g\leq \frac{1}{C}|\alpha|^2dV_g$$
	where $dV_g$ is the volume form of $g$ and the norms are induced by $h$ and $g$.
\end{lem}

Now let $\chi\in C^{\infty}_0(U)$ be a non-negative cut-off function, that equals 1 in a neighborhood of the subset $\{p\in U|\nu(p)\leq \delta+c \}$ for some small $c>0$. Then for each $\tilde{s}\in G_\delta$ with $\parallel \tilde{s}\parallel^2=1$, we take the orthogonal projection of $\chi  \tilde{s}$ onto $\hcal_k$. The way to do so is to solve the $\dbar$-equation
$$\dbar v=\dbar \chi \otimes \tilde{s}.$$
And since $|\tilde{s}|$ is $\epk$ on the place where $\dbar\chi\neq 0$, we can apply the lemma above to see that the global holomorphic section $\chi\tilde{s}-v\in \hcal_k$ has $\epk$ $L^2$-mass outside $U$ and the $L^2$-mass of $\chi\tilde{s}-v-\tilde{s}$ within $U$ is also $\epk$.
So for each $f\in \hcal^{\delta k}_k$ with $\parallel f\parallel^2=1$, we have 
$$|<f,\chi\tilde{s}-v>|=\epk $$
 We then take the orthogonal projection of $\chi\tilde{s}-v$ onto $(\hcal^{\delta k}_k)^\perp$ and denote by $\bar{s}$ the section obtained. So we have 
 $$\int_U |\bar{s}-\tilde{s}|^2e^{-k\phi}\omega^m=\epk.$$ 
We denote by $J$ the map we just constructed from $G_\delta$ to $(\hcal^{\delta k}_k)^\perp$. So $J$ is almost an isometry, namely 
$$\parallel J(\tilde{s})\parallel^2=(1+\epk)\parallel \tilde{s}\parallel^2.$$
Therefore, if we start with an orthonormal basis of $G_\delta$, we obtain an almost orthonormal basis of $(\hcal^{\delta k}_k)^\perp$. And then after a Gram-Schmidt process, we get an orthonormal basis. We denote by $\rho_G$ the density function of $G_\delta$, considered as a function on $X$ by extenstion by zero. Since the density function of $(\hcal^{\delta k}_k)^\perp$ is $\rho_k-\rho_k^{\delta k}$, we have 
$$\int_X (\rho_k-\rho_k^{\delta }-\rho_G)\omega^m=\epk.$$
Since, $\rho_G$ has $\epk$ mass outside  $Y_\delta$, we have
$$\int_{X\backslash Y_\delta} (\rho_k-\rho_k^{\delta })\omega^m=\epk.$$
Recall that by a local construction of the peak section $s_p$ for any $p\in X$(see for example\cite{Donaldson15}), $|s_p|$ decays exponentially away from $p$. More precisely, $|s_p(q)|=\epk$, when $d(p,q)>\frac{\log k}{\sqrt{k}}$. When $\delta$ is small enough, $\nu$ is comparable with the distance function $d^2(p,q)$, so if we denote by 
 $Y'_\delta=\{q|\nu(q)\leq \delta+3(m-n)k^{-1/4} \}$, then for any unit $s\in (\hcal^{\delta k}_k)^\perp$, we have
 $$|<s,s_p>|=\epk,$$
for $p\in X\backslash Y'_\delta$. Therefore $|s(p)|=\epk$, $p\in X\backslash Y'_\delta$. Since $\lim_{k\to \infty }Y'_\delta=\{q|\nu(q)\leq \delta \}$, this implies the second part of theorem \ref{main}.

	\bibliographystyle{plain}

\bibliography{references}

\end{document}